\documentclass[11pt]{amsart}

\usepackage{hyperref}
\usepackage{mathtools}
\usepackage{amsthm}
\usepackage{amssymb}
\usepackage{amsfonts}
\usepackage[utf8]{inputenc}
\usepackage{float}
\usepackage{verbatim}
\usepackage{enumitem}
\usepackage{booktabs}

\usepackage{tikz}

\numberwithin{equation}{section}
\newtheorem{theorem}{Theorem}[section]

\newtheorem{lem}[theorem]{Lemma}
\newtheorem{proposition}[theorem]{Proposition}

\newtheorem{lemma}[theorem]{Lemma}

\theoremstyle{definition}

\theoremstyle{remark}
\newtheorem{remark}[theorem]{Remark}

\newcommand{\eps}{\varepsilon}

\DeclareMathOperator{\area}{\mathrm{Area}}

\newcommand{\Sph}{\mathbb{S}}

\newcommand{\DD}{\mathbb{D}}

\newcommand{\T}{\mathbb{T}}

\newcommand{\R}{\mathbb{R}}

\newcommand{\Isom}{\mathrm{Isom}}

\newcommand{\Met}{\mathrm{Met}}

\definecolor{light-gray}{gray}{.95}
\definecolor{dark-gray}{gray}{.7}

\begin{document}

\title[Existence of $\bar{\lambda}_1$-maximizing metrics]{Existence of metrics maximizing the first Laplace eigenvalue on closed surfaces}
\author[M.~Karpukhin]{Mikhail~Karpukhin}
\author[R.~Petrides]{Romain~Petrides}
\author[D.~Stern]{Daniel~Stern}
\date{}
\address{Department of Mathematics, University College London, 25 Gordon Street, London, WC1H 0AY, UK} \email{m.karpukhin@ucl.ac.uk}
\address{Universit\'e Paris Cit\'e, Institut de Math\'ematiques de Jussieu - Paris Rive Gauche, b\^atiment Sophie Germain, 75205 PARIS Cedex 13, France}
\email{romain.petrides@imj-prg.fr}
\address{Department of Mathematics, Cornell University, Ithaca NY, 14853} \email{daniel.stern@cornell.edu}

\begin{abstract}

Building on seminal work of Nadirashvili and previous work of the authors, we prove the existence of metrics maximizing the area-normalized first eigenvalue of the Laplacian on every closed nonorientable surface, and give a simple new proof of existence in the orientable case complementing that of~\cite{PetridesNew}, thus resolving the long-standing existence problem for $\lambda_1$-maximizing metrics on closed surfaces of any topology. Namely, we prove by contradiction that the supremum $\Lambda_1(M)$ of the normalized first eigenvalue over all metrics on $M$ obeys the strict monotonicity $\Lambda_1(M\#\mathbb{RP}^2)>\Lambda_1(M)$ and $\Lambda_1(M\#\mathbb{T}^2)>\Lambda_1(M)$ under the attachment of cross-caps and handles, via a substantial refinement of techniques introduced in \cite{KKMS}.

\end{abstract}

\maketitle

\section{Introduction}

Let $M$ be a compact surface without boundary, and for a given metric $g$ on $M$, consider the first nonzero eigenvalue $\lambda_1(\Delta_g)$ of the Laplacian $\Delta_g=d_g^*d$. Multiplying by the area gives a scale-invariant quantity
$$
\bar{\lambda}_1(M,g):=\lambda_1(\Delta_g)\area(M,g),
$$
which can be viewed as a functional on the space $\Met(M)$ of Riemannian metrics on $M$. While the infimum of $\bar{\lambda}_1(M,g)$ over metrics on $M$ is easily seen to vanish, a series of results dating back to the 1970s show that the supremum
$$
\Lambda_1(M):=\sup\left\{\bar{\lambda}_1(M,g) \mid g\in \Met(M)\right\}
$$
is not only finite, but realized in many cases by distinguished metrics on $M$~\cite{YangYau, Hersch, NadirashviliT2, LiYau, EGJ, JNP, JLNNP, NaySh}. 

In the 1990s, as an ingredient in the characterization of $\Lambda_1(\mathbb{T}^2)$, Nadirashvili made the deep observation that critical metrics for the functional $\bar{\lambda}_1(M,g)$ are induced by minimal immersions into round spheres by first eigenfunctions of the Laplacian~\cite{NadirashviliT2}. In the decades since, this link between critical metrics for $\bar{\lambda}_1$ and minimal surfaces has served both as a key tool and an important source of motivation in the development of the variational theory for $\bar{\lambda}_1$, and related invariants for higher eigenvalues and other operators~\cite{FSinvent, Petridesk, PetridesSt, PetridesComb, PetridesCombSt,  KNPP1, KRP2, PetTew, LimaMenezes, GPA, KMP}.

In the wake of these developments, attention has turned more recently to the question of \emph{existence} of extremal metrics for $\bar{\lambda}_1$ and related functionals. In~\cite{Petrides1}, building on ideas of Fraser and Schoen~\cite{FSinvent}, Petrides gave a general existence result for metrics maximizing $\bar{\lambda}_1$ within any fixed conformal class on a surface $M$, reducing the global existence question to the question of precompactness of maximizing sequences within the moduli space of conformal structures on $M$. In particular, the results of~\cite{Petrides1} in the orientable setting and~\cite{MS2} for nonorientable surfaces reduce the global existence problem to that of establishing strict monotonicity for the supremum $\Lambda_1(M)$ with respect to the attachment of handles and cross-caps, as follows.

\begin{theorem}[Petrides~\cite{Petrides1}, Matthiesen-Siffert~\cite{MS2}]
\label{gap.suff}

Suppose that 
$$
\Lambda_1(M)>\Lambda_1(M_0)
$$
for any closed surface $M_0$ such that $M\approx M_0\# \mathbb{RP}^2$ or $M\approx M_0\#\T^2$. Then $M$ admits a $\bar{\lambda}_1$-maximizing metric, which is smooth up to finitely many isolated conical singularities.
\end{theorem}

While the non-strict inequalities 
\[
\Lambda_1(M_0\#\mathbb{RP}^2)\geq \Lambda_1(M_0), \quad
\Lambda_1(M_0\# \T^2)\geq \Lambda_1(M_0)
\]
have been known for some time (see e.g.~\cite{CES}), proving the strict inequality for all topologies turns out to be a more challenging problem. One natural strategy would be to proceed by an induction on topological complexity of the surface, assuming the existence of a $\bar{\lambda}_1$-maximizing metric $g_{\max}$ on the surface $M_0$, then attaching to $(M_0,g_{\max})$ a handle or cross-cap with carefully chosen geometry in such a way that the resulting metric $g'$ on $M_0\# \mathbb{RP}^2$ or $M_0\# \T^2$ can be shown to satisfy $\bar{\lambda}_1(g')>\bar{\lambda}_1(M_0,g)$. For such constructions, however, directly estimating the difference $\bar{\lambda}_1(g')-\bar{\lambda}_1(M_0,g)$ is in practice quite delicate, and previous attempts to carry out this strategy by attaching handles and cross-caps with a hyperbolic structure as in~\cite{MatthiesenSiffert} are unfortunately known to contain gaps. 

In~\cite{KKMS}, an alternative approach was developed for establishing strict monotonicy of $\Lambda_1$ under handle attachment, where it was applied successfully in the presence of certain discrete symmetry groups, and used to obtain the partial result $\Lambda_1(M_0\# \T^2\#\T^2)>\Lambda_1(M_0)$ for the original $\bar{\lambda}_1$-maximization problem. In this approach, rather than trying to estimate $\bar{\lambda}_1$ directly for special metrics on $M_0\#\T^2$, a crude attachment of flat handles is used instead to supply a conformal class $[g']$ on $M_0\#\T^2$, for which one uses the estimate 
$$
\Lambda_1(M_0\#\T^2,[g']):=\sup_{h\in [g']}\bar{\lambda}_1(h)\leq \Lambda_1(M_0\#\T^2)\leq \Lambda_1(M_0)
$$
to obtain a contradiction. Roughly speaking, applying this estimate for a certain shrinking family of handles attached at a prescribed pair of points, the results of~\cite{KSjems} and a collection of elementary estimates are then used to deduce the existence of sphere-valued maps $F\colon (M_0\#\T^2,g')\to \mathbb{S}^n$ with small energy in the handle region, that limit in a $W^{1,2}$-sense to a map $\Phi\colon M_0\to \mathbb{S}^n$ by first eigenfunctions for $\Delta_{g_{\max}}$ identifying the pair of points in $M$ at which the handle was attached. In many cases, this leads to the desired contradiction.

Recently, the existence theory for $\bar{\lambda}_1$-maximizing metrics on oriented surfaces was completed in~\cite{PetridesNew}, where the gap $\Lambda_1(M_0\#\T^2)>\Lambda_1(M)$ is obtained by a contradiction argument, combining ideas from~\cite{KKMS} with a delicate application of the Ekeland variational principle for a carefully chosen $\bar{\lambda}_1$-maximizing sequence on $M_0\#\T^2$. As in~\cite{KKMS}, the contradiction comes from the existence of sphere-valued maps $\Phi_{p,q}\colon M_0\to \mathbb{S}^n$ by first eigenfunctions identifying a prescribed pair of points $p,q\in M_0$; the key new feature is that these maps can now also be taken to be \emph{weakly conformal}. In particular, fixing $p$ and letting $q\to p$ in this construction, one obtains a weakly conformal map $\Phi\colon M_0\to \mathbb{S}^n$ by first eigenfunctions for which $d\Phi(p)=0$, contradicting the identity $|d\Phi|^2\equiv \lambda_1$ that holds for all such maps. Moreover, this argument also has a natural generalization to other eigenvalue functionals; see~\cite{PetridesNew} for details.

Evidently, the two-point arguments of~\cite{KKMS} and~\cite{PetridesNew} are tailored to handle-attachment constructions, and lack an immediate analog for the attachment of M\"obius bands sufficient to complete the existence theory for nonorientable surfaces. In the present paper, by improving on the techniques developed in~\cite{KKMS}, we succeed in proving the desired monotonicity $\Lambda_1(M_0\#\mathbb{RP}^2)>\Lambda_1(M_0)$ under cross-cap attachment, and use similar ideas to give a simple, self-contained new proof of the handle monotonicity $\Lambda_1(M_0\#\T^2)>\Lambda_1(M_0)$ for all closed surfaces, completing the existence theory for $\bar{\lambda}_1$-maximizing metrics.

\begin{theorem}\label{gap.thm}
For any closed surface $M$, we have $\Lambda_1(M\#\mathbb{RP}^2)>\Lambda_1(M)$ and $\Lambda_1(M\#T^2)>\Lambda_1(M)$. In particular, every closed surface admits a $\bar{\lambda}_1$-maximizing metric, which is smooth up to finitely many conical singularities.
\end{theorem}

In both the orientable and nonorientable case, assuming that the desired strict monotonicity fails, the key idea is to upgrade the $0$-th order estimates used to prove the identity $\Phi(p)=\Phi(q)$ for the limiting eigenmap $\Phi$ in~\cite{KKMS} to \emph{first-order estimates} at a point of attachment, showing directly the existence of a sphere-valued eigenmap with $d\Phi(p)=0$, giving a contradiction. To this end, an essential new ingredient is the introduction of refined extension operators
from the M\"obius band and cylinder into the disk that distinguish between even and odd functions with respect to suitable symmetries, replacing the simple harmonic extension used in~\cite{KKMS, PetridesNew}. When attaching cross-caps, this allows us in particular to remove certain odd Fourier modes from key estimates in small disks near the point of attachment, which together with the techniques of~\cite{KKMS} and an elementary bootstrapping argument allows us to deduce the desired vanishing $d\Phi(p)=0$ for a limiting sphere-valued eigenmap. To mimic this argument for handle attachment, we fix a point $p$ and attach handles of width $\epsilon$ at a point $q$ of distance $\sqrt{\epsilon}$ from $p$, with a \emph{maximally twisted} attaching map at the boundary circles. In the final argument, $0$-th order estimates refining those of~\cite{KKMS} lead to the vanishing of $d\Phi(p)$ in the $q-p$ direction, while the extra 180$^{\circ}$ twist in the attaching map and the properties of the refined extension operator allow us to implement a bootstrapping argument similar to that used in the cross-cap case, to conclude that $d\Phi(p)$ vanishes in the direction perpendicular to $q-p$ as well.

With the existence theory for the $\bar{\lambda}_1$-maximization problem settled, it is natural to ask whether the techniques of the present paper generalize to higher eigenvalues and other related optimization problems. Regarding higher eigenvalues, first note that the analog of Theorem~\ref{gap.suff} does not hold in general, as it is possible for a $\bar{\lambda}_k$-maximizing sequence to degenerate to a disconnected union of extremal geometries for lower eigenvalues, as is known to occur for $\mathbb{S}^2$ and $\mathbb{RP}^2$~\cite{KNPP1,KRP2}. Nonetheless, strict topological monotonicity for the supremum $\Lambda_k(M)$ of $\bar{\lambda}_k(M,g)$ would allow one to rule out some degenerations for maximizing sequences, and in the case of oriented surfaces, the results of~\cite{PetridesNew} show that strict monotonicity under handle attachment continues to hold in this setting. 

Ultimately, we expect that a variant of the techniques in the present paper can be employed to prove strict monotonicity of $\Lambda_k(M)$ under cross-cap attachment as well, but for the moment we are limited to the first and perhaps second eigenvalues. This limitation comes from the essential role in our estimates played by Proposition \ref{balanced.f} below, which relies in turn on the characterization of conformal suprema for $\bar{\lambda}_1$ via certain min-max energies in~\cite{KSjems}. While~\cite{KSjems} gives an analog of this result for $\bar{\lambda}_2$, the problem of obtaining analogous min-max characterizations for the conformal suprema of $\bar{\lambda}_k$--and in particular, the analog of Proposition \ref{balanced.f}--remains open at present for $k\geq 3$.

Given that the techniques of~\cite{KKMS} and~\cite{PetridesNew} also apply in certain cases to the Steklov optimization problem considered in ~\cite{FSinvent}, one might also ask whether the methods of the present paper carry over to that setting, to complete the existence theory for $\bar{\sigma}_1$-maximizing metrics on compact surfaces with boundary. To this end, we note that the problem left open by~\cite{KKMS, PetridesNew} is that of proving strict monotonicity of the supremum $\Sigma_1(N)=\sup \bar{\sigma}_1(N,g)$ under attachment of strips at pairs of \emph{distinct} components of $\partial N$. Unfortunately, this problem does not seem amenable to a straightforward adaptation of the methods of Section 4 below, which rely on the attachment of handles at pairs of points which are arbitrarily close to one another. Nonetheless, we are optimistic that a further refinement of the ideas in~\cite{KKMS, PetridesNew}, and the present paper may lead to the resolution of the existence problem for $\bar{\sigma}_1$- maximizing metrics in full generality in the near future.

\subsection*{Acknowledgements} The research of D. S. is supported by NSF grant DMS 2404992 and the Simons Foundation award MPS-TSM-00007693.

\section{Conformal classes on connected sums from singular gluing constructions}\label{glue}

Given a pair of compact Riemannian surfaces $(N^2,g)$ and $(\Gamma^2,h)$ with nonempty boundaries, and a diffeomorphism
$$
\phi\colon \partial \Gamma\to \partial N,
$$
consider the closed surface 
$$
M'=N\cup_{\phi}\Gamma
$$
given by identifying $\partial N$ and $\partial\Gamma$ via $\phi$, equipped with the $L^{\infty}$ metric $g'$ for which $g'|_N=g$ and $g'|_{\Gamma}=h$. 

Metrics of the form $g'$ are often used as test metrics for the $\bar{\lambda}_1$-maximization problem on connected sums, where $(N,g)$ is the complement of one or two disks on a closed surface $M$ and $(\Gamma,h)$ is a flat M\"obius band or a cylinder, see e.g.~\cite{MS1}. In~\cite{KKMS} and in the proof of Theorem~\ref{gap.thm} below, rather than using the metrics $g'$ themselves as test metrics, we work at the level of conformal classes, through the estimate
\begin{equation}\label{cf.bd}
\Lambda_1(M',[g'])\leq \Lambda_1(M),
\end{equation}
where $\Lambda_1(M',[g'])$ denotes the supremum of all smooth metrics in the conformal class associated to $g'$. 

In either case, to justify these arguments, it is important to observe that the conformal class $[g']$ is indeed well-defined: namely, $g'$ can be written in the form $\rho g_0$ for a smooth metric $g_0$ on $M'$ and an $L^{\infty}$ function $\rho\colon M'\to [0,\infty)$. This is a straightforward exercise in the geometry of surfaces, but we include details for the sake of completeness.

\begin{lemma}\label{straight.lem}
For the singular glued metric $g'$ on $M'$ defined as above, there is a smooth metric $g_0$ on $M'$ and a positive $L^{\infty}$ function $\rho: M'\to (0,\infty)$ with respect to which $g'=\rho g_0$. 
\end{lemma}

\begin{proof}
On a small tubular neighborhood $U\supset\partial N$ of the boundary in $(N,g)$, writing the metric in Fermi coordinates, it is easy to see that there exists a positive smooth function $\psi: U\to (0,\infty)$ depending on the distance to $\partial N$ such that $(U,\psi g)$ is isometric to the cylinder $\partial N\times [0,\epsilon)$ for some $\epsilon>0$. Fixing a nonnegative cutoff function $\chi \in C_c^{\infty}(U)$ such that $\chi\equiv 1$ in a neighborhood of $\partial N$, we then set $f_1:=\chi \psi+(1-\chi)\in C^{\infty}(N,\mathbb{R}_{>0})$. In this way, we get a conformal metric $(N,f_1g)$ isometric to a cylinder over $\partial N$ in a neighborhood of the boundary. 

Next, note that diffeomorphisms between one-dimesional manifolds is always conformal.s
Therefore it is straightforward to find a smooth conformal factor $f_2\colon \Gamma\to (0,\infty)$ such that $\phi\colon \partial \Gamma\to \partial N$ is an isometry with respect to the metric $f_2h$ on $\Gamma$ and the metric $f_1g$ on $N$. Starting from this conformal metric $(\Gamma ,f_2h)$, we can then argue just as in the preceding paragraph to obtain finally a conformal factor $f_3: \Gamma \to (0,\infty)$ with flat cylindrical geometry near each boundary component, and such that $\phi\colon \partial \Gamma\to \partial N$ is an isometry with respect to $f_3h$ and $f_1g$. 

It's then straightforward to check that the metric $g_0$ on $M'=N\cup_{\phi}\Gamma$ restricting to $f_1g$ on $N$ and $f_3h$ on $\Gamma$ is a smooth metric on $M'$, so defining $\rho\in L^{\infty}(M')$ by setting $\rho=\frac{1}{f_1}$ on $N$ and $\rho=\frac{1}{f_3}$ on $\Gamma$ completes the proof.

\end{proof}

Though we will not need it for our arguments, we note that the use of $g'$ as a test metric for the $\bar{\lambda}_1$-maximization problem on $M'$ can then easily be justified by combining Lemma \ref{straight.lem} with the results of~\cite{Kok.var, KSjems, GKL}. Namely, recall that for any conformal class $[g_0]$ on $M'$ and any nontrivial Radon measure $\mu$ on $M'$, we can define the first eigenvalue
$$
\lambda_1([g_0],\mu):=\inf\left\{\left.\frac{\|df\|_{L^2(M',[g_0])}^2}{\|f\|_{L^2(\mu)}^2}\right| \int f d\mu=0\text{ and }f\neq 0\right\}
$$
and its normalization
$$
\bar{\lambda}_1([g_0],\mu):=\mu(M')\lambda_1([g_0],\mu).
$$
By Lemma \ref{straight.lem}, one can then interpret $\bar{\lambda}_1(M',g')$ as
$$
\bar{\lambda}_1(M',g')=\bar{\lambda}_1([g_0],\rho dv_{g_0}),
$$
and results of~\cite{Kok.var, KSjems, GKL} imply that $\bar{\lambda}_1(M',g')\leq \Lambda_1(M',[g_0])$. 

More generally,~\cite[Theorem 1.4]{KSjems} shows that $\bar{\lambda}_1([g_0],\mu)\leq \Lambda_1(M,[g_0])$ whenever $\mu$ belongs to a large class of measures for which the natural map $C^{\infty}(M')\to L^2(\mu)$ extends to a compact map $W^{1,2}(M')\to L^2(\mu)$. In fact, this statement can be extended to any measure defining a bounded linear functional on $W^{1,2}(M')$, by an application of the following useful proposition, which will also play an important role in the proofs of Theorem~\ref{gap.thm}.

\begin{proposition}[cf. Proposition 3.1 in~\cite{KNPS}]
\label{balanced.f}
Given a conformal class $[g]$ on $M'$, there exists an integer $n=n([g])\in \mathbb{N}$ such that for every $\beta\in W^{1,2}(M')^*$, there exists a map $F\in W^{1,2}(M',\mathbb{S}^n)$ such that 
$$
\beta(F)=0\in \mathbb{R}^{n+1}
$$
and
$$
\int |dF|_{g}^2\,dv_{g}\leq \Lambda_1(M,[g]).
$$
\end{proposition}
\begin{proof}
For $\beta\in W^{1,2}(M')^*=W^{-1,2}(M')$ defined by integration against an admissible measure, this is an immediate consequence of [KNPS Proposition 3.1] and its proof. Given an arbitrary $\beta\in W^{-1,2}(M')$, we can find a sequence $\beta_j\in W^{1,2}(M')^*$ given by integration against smooth functions such that
$$
\lim_{j\to\infty}\|\beta_j-\beta\|_{W^{-1,2}(M')}=0,
$$
and since the desired statement holds in the smooth case, there exists an associated sequence of maps $F_j\in W^{1,2}(M',\mathbb{S}^n)$ satisfying
$$
\beta_j(F_j)=0\in \mathbb{R}^{n+1}
$$
and
$$
\int |dF_j|_g^2\,dv_g\leq \Lambda_1(M',[g]).
$$
Passing to a subsequence, we can then extract a limiting map $F\in W^{1,2}(M',\Sph^n)$ such that $F_j$ converges to $F$ weakly in $W^{1,2}(M',\mathbb{S}^n)$. The bound
$$
\int |dF|_g^2\,dv_g\leq \Lambda_1(M',[g])
$$
then follows immediately from lower semicontinuity of energy, while the weak convergence $F_j\rightharpoonup F$ in $W^{1,2}$ and strong convergence $\beta_j\to \beta$ in $W^{-1,2}$ give
$$
\beta(F)=\lim_{j\to\infty}\beta(F_j)=\lim_{j\to\infty}\beta_j(F_j)=0,
$$
so $F$ provides the desired map.
\end{proof}

\section{Attaching cross-caps}

In this section we prove the following gap result, providing the main ingredient for the nonorientable part of Theorem \ref{gap.thm}.

\begin{theorem}\label{cap.gap}
Let $M$ be a closed surface, and let $g_0$ be a (possibly degenerate) metric on $M$ of the form $g_0=fg_1$, where $g_1$ is a smooth metric and $f$ is a smooth, nonnegative function with $\#f^{-1}\{0\}<\infty$. Then
$$
\bar{\lambda}_1(M,g_0)=\bar{\lambda}_1([g_1],f\,dv_{g_1})<\Lambda_1(M\#\mathbb{RP}^2).
$$
\end{theorem}

We prove Theorem \ref{cap.gap} by identifying a family of conformal classes $[g']$ on $M'=M\#\mathbb{RP}^2$ for which the assumption $\Lambda_1(M',[g'])\leq \bar{\lambda}_1(M,g)$ leads to a contradiction. To this end, fix a point $p\in M$, and fix a reference metric $g_p\in [g_1]$ which is Euclidean in a neighborhood of $p$. For a small $\epsilon>0$, we then consider the manifold with boundary
$$N_{p,\epsilon}:=M\setminus D_{\epsilon}(p)$$
given by removing the $g_p$-geodesic disk of radius $\epsilon$ centered at $p$.

Next, denote by $\Gamma_L:=\Sph^1\times [-L,L]/\sim$ the flat M\"obius band of width $2\pi$ and length $L$, where $(z,t)\sim (-z,-t)$, and let $\phi\colon \partial \Gamma_L\to \partial N_{p,\epsilon}$ be a homothety with a scaling factor of $\epsilon$, correponding to $\phi(z,L)=\epsilon z$ under the identification $D_{\epsilon}(p)\cong D_{\epsilon}(0)$. As in Section \ref{glue}, we then define an $L^{\infty}$ metric $g'_{p,\epsilon,L}$ on 
$$
M'=N_{p,\epsilon}\cup_{\phi}\Gamma_L\approx M\#\mathbb{RP}^2
$$
such that $g'_{p,\epsilon,L}$ agrees with $g_p$ on $N_{p,\epsilon}$ and $g'_{p,\epsilon,L}$ agrees with the given flat metric on $\Gamma_L$. The main technical result of this section can then be stated as follows.

\begin{proposition}\label{cap.prop}
If $\Lambda_1(M',[g'_{p,\epsilon,L}])\leq \bar{\lambda}_1(M,g_0)$ for all $L\in (0,\infty)$ and all $\epsilon$ sufficiently small, then $f(p)=0$. 
\end{proposition}

In particular, by choosing a point $p$ for which $f(p)>0$, we see that $\Lambda_1(M',[g'_{p,\epsilon,L}])>\bar{\lambda}_1(M,g_0)$ for some choice of $\epsilon$ and $L$, from which Theorem \ref{cap.gap} follows. The remainder of the section is devoted to the proof of Proposition \ref{cap.prop}.

\begin{remark}\label{eigen.basics}
Though the conformal metric $g_0=fg_1=\tilde{f}g_p$ fails to define a classical metric at points where $f$ vanishes, we remind the reader that the compactness of the map $W^{1,2}(M,g_1)\to L^2(M,dv_{g_0})$ implies the existence of a sequence $0=\lambda_0(M,g_0)<\lambda_1(M,g_0)\leq \cdots \to\infty$ of eigenvalues and an orthonormal sequence of eigenfunctions $\phi_i\in L^2(M,dv_{g_0})$ satisfying
$$
\Delta_{g_1}\phi_i=\lambda_i f \phi_i
$$
defining a Schauder basis for $W^{1,2}(M)$. Moreover, smoothness of these eigenfunctions $\phi_i$ follows easily from the smoothness of $f$ and standard elliptic regularity theory.
\end{remark}

\subsection{A refined extension operator for cross-caps}

In the arguments of~\cite{KKMS} and~\cite{PetridesNew}, a key ingredient is the following lemma, providing sharp bounds for the harmonic extension operator from the cylinder $C_L=\Sph^1\times [0,L]$ to the unit disk $\DD$.

\begin{lemma}[Lemma 8.11 in~\cite{KKMS}]
\label{hex.lem}
Let $L\geq \frac{3}{2}\log(2)$, and let 
$$
H\colon W^{1,2}(C_L)\to W^{1,2}(\mathbb{D}^2)
$$ 
be the map sending $u\in W^{1,2}(C_L)$ to the harmonic extension $H(u)=\hat{u}$ of $u(\cdot,L)$ to the unit disk $\mathbb{D}^2$. Then
$$
\|d H(u)\|_{L^2(\mathbb{D}^2)}^2\leq \left(1+Ce^{-2L}\right)\|du\|_{L^2(C_L)}^2.
$$
\end{lemma}

Note that any function $u\in W^{1,2}(\Gamma_L)$ on the M\"obius band $\Gamma_L$ can be identified with a function $u\in W^{1,2}(C_L)$ satisfying $u(z,0)=u(-z,0)$, so that the corresponding harmonic extension operator $H\colon W^{1,2}(\Gamma_L)\to W^{1,2}(\mathbb{D}^2)$ likewise satisfies the estimate given in Lemma \ref{hex.lem}. 

Estimates of this form were sufficient for the results of [KKMS, Thm 8.1], but to prove Proposition~\ref{cap.prop}, we will need finer estimates, requiring a more delicate choice of extension operator $K\colon W^{1,2}(\Gamma_L)\to W^{1,2}(\mathbb{D}^2)$ tailored to the M\"obius band $\Gamma_L$. To this end, consider the orthogonal decomposition
\[
W^{1,2}(\Gamma_L)=\mathcal{E}\oplus \mathcal{O},
\]
of $W^{1,2}(\Gamma_L)$ into the even functions
$$
\mathcal{E}:=\{u\in W^{1,2}(\Gamma_L)\mid u(-z,t)=u(z,t)\}
$$
and odd functions
$$
\mathcal{O}:=\{u\in W^{1,2}(\Gamma_L)\mid u(-z,t)=-u(z,t)\}
$$
for the rotation $(z,t)\mapsto (-z,t)$. Crucially, note that for any $u\in \mathcal{O}$, since $(z,0)\sim (-z,0)$, we must have
\begin{equation}\label{odd.van}
u(z,0)=u(-z,0)=-u(z,0)=0\text{ for any }z\in \Sph^1.
\end{equation}
Letting $\pi_E\colon W^{1,2}(\Gamma_L)\to \mathcal{E}$ and $\pi_O\colon W^{1,2}(\Gamma_L)\to\mathcal{O}$ denote the obvious projections, we now build the refined extension operator as follows.

\begin{lem}\label{ext.bds}
For $L\geq \frac{3}{2}\log(2)$, there exists an operator $K\colon W^{1,2}(\Gamma_L)\to W^{1,2}(\DD^2)$ such that 
$$
K(u)(z)=u(z,L)\text{ for }z\in \Sph^1,
$$
$$
\|K(u)\|_{L^{\infty}(\DD^2)}\leq 2\|u\|_{L^{\infty}(\Gamma_L)},
$$
and writing $u_E=\pi_E(u)$ and $u_O=\pi_O(u)$, we have
\[
\|d[K(u_E)]\|_{L^2(\DD^2)}^2\leq (1+Ce^{-2L})\|d(u_E)\|_{L^2(\Gamma_L)}^2
\]
and
\[
\|d[K(u_O)]\|_{L^2(\DD^2)}^2=\|d(u_O)\|_{L^2(\Gamma_L)}^2.
\]
Moreover, $K(u_E)$ is even and $K(u_O)$ is odd with respect to the antipodal map $z\mapsto -z$ in the disk, and $K(u_O)\equiv 0$ on $D_{e^{-L}}(0)$.
\end{lem}
\begin{proof}
First, and most importantly, we define the operator $K\colon \mathcal{O}\to W^{1,2}(\DD^2)$ on the odd functions $\mathcal{O}$. On the annulus $A_L=\DD\setminus D_{e^{-L}}(0)$, note that the map 
$$
F\colon A_L\to \Gamma_L,\text{ }F(z):=\left(\frac{z}{|z|}, L+\log |z|\right)
$$
is conformal, and satisfies $F(z)=(z,L)$ for $z\in \Sph^1$ and $F(\partial D_{e^{-L}}(0))=(\mathbb{S}^1\times \{0\})/\sim$. In particular, if $u\in \mathcal{O}$, it follows from \eqref{odd.van} that $u\circ F$ vanishes on $\partial D_{e^{-L}}(0)$, and defining
$$
K(u)=u\circ F\text{ on }A_L\text{ and }K(u)=0\text{ on }D_{e^{-L}}(0),
$$
it follows that $K$ gives a well-defined map $\mathcal{O}\to W^{1,2}(\DD)$. Moreover, it is clear from the conformality of $F$ that
$$
\|d[K(u)]\|_{L^2}^2=\|du\|_{L^2}^2
$$
for $u\in \mathcal{O}$, while the fact that $K(u)(z)=u(z,L)$ for $z\in \Sph^1$ and the bound $\|K(u)\|_{L^{\infty}}\leq \|u\|_{L^{\infty}}$ likewise follow immediately from the definition, as does the fact that
$$
K(u)(-z)=-K(u)(z)\text{ for any }u\in \mathcal{O}.
$$

We then extend $K$ to an operator $K: W^{1,2}(\Gamma_L)\to W^{1,2}(\DD)$ by setting
$$
K(u):=H(u_E)+K(u_O),
$$
where $H$ is the harmonic extension operator as in Lemma \ref{hex.lem}. Combining the preceding discussion with Lemma \ref{hex.lem} and familiar properties of the harmonic extension $W^{\frac{1}{2},2}(\Sph^1)\to W^{1,2}(\DD)$, it is now easy to see that $K$ satisfies all the desired properties.
\end{proof}

As a key consequence of Lemma \ref{ext.bds}, note that the operator $K$ satisfies
\begin{equation}\label{ener.drop}
\|d(K(u))\|_{L^2(\DD^2)}^2-\|du\|_{L^2(\Gamma_L)}^2=\|d(K(u)_E)\|_{L^2(\DD^2)}^2-\|du_E\|_{L^2(\Gamma_L)}^2,
\end{equation}
where we denote by $K(u)_E$ the even part of $K(u)$ with respect to the rotation $z\mapsto -z$. This is an immediate consequence of the estimates for $\|d(K(u_E))\|_{L^2}$ and $\|d(K(u_O))\|_{L^2}$, together with the observation that $K(u_E)=K(u)_E$ and $K(u_O)=K(u)_O$ are orthogonal in $W^{1,2}(\DD^2)$. 

\begin{remark}\label{vect.valued}
Though the spaces $W^{1,2}(\Gamma_L)$ and $W^{1,2}(\DD^2)$ may be interpreted as spaces of real-valued Sobolev functions, note that these could be replaced with the vector-valued Sobolev spaces $W^{1,2}(\Gamma_L,\mathbb{R}^n)$ and $W^{1,2}(\DD^2,\mathbb{R}^n)$ for arbitrary $n$, with no changes in the conclusions of Lemma \ref{ext.bds}.
\end{remark}

\subsection{Proof of Proposition \ref{cap.prop}}

Returning to the setting of Theorem \ref{cap.gap} and Proposition \ref{cap.prop}, we next reduce the proof of Proposition \ref{cap.prop} to the following lemma.

\begin{lemma}\label{cap.grad.small}
In the setting of Theorem \ref{cap.gap} and Proposition \ref{cap.prop}, there exists $L=L_0\in (1,\infty)$ and $C<\infty$ independent of $\epsilon>0$ such that if $\Lambda_1(M',[g'_{p,\epsilon,L}])\leq \bar{\lambda}_1(M,g_0)$, then there exists a map $\Phi_{\epsilon}\colon M\to \mathbb{R}^{n+1}$ satisfying
$$
\Delta_{g_1}\Phi_{\epsilon}=\lambda_1(M,g_0)f \Phi_{\epsilon},
$$
$$
\|1-|\Phi_{\epsilon}|\|_{L^2(M,g_1)}^2\leq C\epsilon^2
$$
and
$$
|d\Phi_{\epsilon}(p)|^2\leq C\epsilon.
$$
\end{lemma}

Before proving the lemma, we first show how it can be used to prove Proposition \ref{cap.prop}.

\begin{proof}[Proof of Proposition \ref{cap.prop}]

Suppose that $\Lambda_1(M',[g'_{p,\epsilon,L}])\leq \bar{\lambda}_1(M,g_0)$ for all $L\in (0,\infty)$ and $\epsilon\in (0,\epsilon_0)$, so that for each $\epsilon\in (0,\epsilon_0)$, there exists a map $\Phi_{\epsilon}\colon M\to \mathbb{R}^{n+1}$ satisfying the conclusions of Lemma \ref{cap.grad.small}. Per Remark \ref{eigen.basics}, recall that the first eigenspace for $\Delta_{g_0}$ is a finite-dimensional subspace of $C^{\infty}(M)$, so it follows that we can find a subsequence $\epsilon_j\to 0$ such that the maps $\Phi_{\epsilon_j}$ converge in $C^1$ to a limit $\Phi\colon M\to \mathbb{R}^{n+1}$ satisfying
$$
\Delta_{g_1}\Phi=\lambda_1(M,g_0)f \Phi,\qquad |\Phi|\equiv 1,$$
and
$$
d\Phi(p)=0.
$$
Standard computations for sphere-valued maps then give
$$
0=\frac{1}{2}\Delta_{g_1}|\Phi|^2=\Delta_{g_1}\Phi\cdot \Phi-|d\Phi|_{g_1}^2=\lambda_1(M,g_0)f-|d\Phi|_{g_1}^2,
$$
which together with the vanishing $d\Phi(p)=0$ implies that
$$
f(p)=\lambda_1(M,g_0)^{-1}|d\Phi(p)|_{g_1}^2=0,
$$
as desired.
\end{proof}

The remainder of the section is devoted to the proof of Lemma \ref{cap.grad.small}. To begin, recalling the construction of $(M',g'_{p,\epsilon,L})$, observe that we can use the operator $K$ of Lemma \ref{ext.bds} and the identification $\phi\colon \partial \Gamma_L\to \partial D_{\epsilon}(p)\subset M$ to define a map
$$
K\colon W^{1,2}(\Gamma_L)\to W^{1,2}(D_{\epsilon}(p))
$$
satisfying energy estimates and $L^{\infty}$ bounds identical to those of Lemma \ref{ext.bds}, alongside the trace condition
$$
K(u)(\phi(z))=u(z)
$$
for $z\in \partial \Gamma_L=\mathbb{S}^1\times L$. In particular, in this way we can then define a bounded linear map
$$
T\colon W^{1,2}(M')\to W^{1,2}(M)
$$
by setting
$$
T(u)(x)=u(x)\text{ for }x\in N_{p,\epsilon}
$$
and
$$
T(u)(x)=K(u|_{\Gamma_L})(x)\text{ for }x\in D_{\epsilon}(p).
$$
Denoting by $u_E$ and $u_O$ the even and odd parts of $u|_{\Gamma_L}$ as in Lemma \ref{ext.bds}, it then follows from the definition of $T$ and \eqref{ener.drop} that
\begin{equation}\label{t.ener.drop}
\|d[T(u)]\|_{L^2(M)}^2=\|du\|_{L^2(M')}^2+\|d[T(u)]_E\|_{L^2(D_{\epsilon}(p))}^2-\|du_E\|_{L^2(\Gamma_L)}^2,
\end{equation}
where $T(u)_E:=\frac{1}{2}(T(u)(p+z)+T(u)(p-z))$ for $p+z\in D_{\epsilon}(p)$, and
\begin{equation}\label{loc.even}
T(u)=T(u)_E\text{ on }D_{e^{-L}\epsilon}(p).
\end{equation}

\begin{proof}[Proof of Lemma \ref{cap.grad.small}]
Most objects defined in the following proof depend on the choice of $\epsilon$ and $L$, for example the map $F$ and all the notions derived from it. To avoid lengthy subscripts, we do not emphasise this dependence in our choice of notation for these objects. Instead,
 we keep careful track on the parameters in all the inequalities, in particular, the constants $C,C', C''$ are always uniform in $\epsilon\in(0,\epsilon_0)$ and $L\in(1,+\infty)$. As is customary, the exact value of the uniform constants changes from line to line.

With $T\colon W^{1,2}(M')\to W^{1,2}(M)$ defined as above, consider the functional $\beta\in W^{-1,2}(M')$ given by
$$
\beta(u):=\int_M T(u)\,dv_{g_0}=\int_M T(u)f \,dv_{g_1}.
$$
It follows from Proposition~\ref{balanced.f} that there exists $n=n([g_1])$ and a map $F\in W^{1,2}(M',\mathbb{S}^n)$ for which 
\begin{equation}
\label{f.ener.bd}
\int_{M'}|dF|^2_{g'}\,dv_{g'}\leq \Lambda_1(M',[g']),
\end{equation}
while $\hat{F}:=T(F)\in W^{1,2}(M,\R^{n+1})$ satisfies
\begin{equation}\label{cap.bal}
\int_M \hat{F}\,dv_{g_0}=0.
\end{equation}
Since the original map $F$ takes values in $\mathbb{S}^n$, note moreover that $|\hat{F}|=|F|=1$ on $M\setminus D_{\epsilon}(p)$, while $|\hat{F}|\leq 2$ on $D_{\epsilon}(p)$, by Lemma \ref{ext.bds}.

By an application of the standard $W^{1,1}_0\to L^1$ Poincar\'e inequality to the function $1-|\hat{F}|^2$ on the disk $D_{\epsilon}(p)$, we observe next that
$$
\int_{D_{\epsilon}(p)}|1-|\hat{F}|^2|\,dv_{g_p}\leq C\epsilon \|\hat{F}\cdot d\hat{F}\|_{L^1(D_{\epsilon}(p))}\leq C'\epsilon^2\|d\hat{F}\|_{L^2(D_{\epsilon}(p))},
$$
which together with the estimates of Lemma \ref{ext.bds} leads to an estimate of the form
$$
\int_{D_{\epsilon}(p)}|1-|\hat{F}|^2|\,dv_{g_0}\leq C''\epsilon^2\|dF\|_{L^2(\Gamma_L)}.
$$
In particular, since $|\hat{F}|=|F|=1$ on $M\setminus D_{\epsilon}(p)$, it follows that
\[
\int_M |\hat{F}|^2\,dv_{g_0}\geq \area(M,g_0)-C\eps^2\|dF\|_{L^2(\Gamma_L)}.
\]

On the other hand, under the hypotheses of Proposition \ref{cap.prop}, we have that $\Lambda_1(M',[g'])\leq \bar{\lambda}_1(M,g_0)$, which together with \eqref{t.ener.drop} and \eqref{f.ener.bd} yields
$$
\int_M|d\hat{F}|_{g_1}^2\,dv_{g_1}\leq \bar{\lambda}_1(M,g_0)+\|d\hat{F}_E\|_{L^2(D_{\epsilon}(p))}^2-\|dF_E\|_{L^2(\Gamma_L)}^2,
$$
and combining this with the preceding estimate, it follows that
\begin{equation}\label{cap.q.1}
Q_{g_0}(\hat{F},\hat{F})\leq \|d\hat{F}_E\|_{L^2(D_{\epsilon}(p))}^2-\|dF_E\|_{L^2(\Gamma_L)}^2+C'\epsilon^2\|dF\|_{L^2(\Gamma_L)},
\end{equation}
where
$$
Q_{g_0}(u,v):=\langle du,dv\rangle_{L^2(M,[g_1])}-\lambda_1(M,g_0)\langle u,v\rangle_{L^2(M,dv_{g_0})}.
$$

Arguing as in [KNPS, Section 2], we observe next that since $\hat{F}$ satisfies the balancing condition \eqref{cap.bal}, and, by definition of $\lambda_1(M,g_0)$, the quadratic form $Q_{g_0}$ is nonnegative definite on the orthogonal complement of the constants in $L^2(M,g_0)$, the Cauchy-Schwarz inequality for $Q_{g_0}$ gives
$$
Q_{g_0}(\hat{F},u)\leq \sqrt{Q_{g_0}(\hat{F},\hat{F})Q_{g_0}(u,u)}.
$$
Now, if $k$ is the largest integer for which $\lambda_k(M,g_0)=\lambda_1(M,g_0)$, in the notation of Remark \ref{eigen.basics}, let $\Phi$ be the projection of $\hat{F}$ onto the eigenspace $\mathcal{E}_{\lambda_1}=\mathrm{Span}\{\phi_1,\ldots,\phi_k\}$, and observe that
$$
Q_{g_0}(\hat{F},\hat{F}-\Phi)=Q_{g_0}(\hat{F}-\Phi,\hat{F}-\Phi)\geq \left(1-\frac{\lambda_1(M,g_0)}{\lambda_{k+1}(M,g_0)}\right)\|d(\hat{F}-\Phi)\|_{L^2}^2.
$$
Combining this with the Cauchy-Schwarz inequality above for $u=\hat{F}-\Phi$, it follows in particular that
$$
\|\hat{F}-\Phi\|_{W^{1,2}}^2\leq C(M,g_0)\sqrt{Q_{g_0}(\hat{F},\hat{F})}\|\hat{F}-\Phi\|_{W^{1,2}},
$$
which together with \eqref{cap.q.1} gives an estimate of the form
\begin{equation}\label{h1close}
\|\hat{F}-\Phi\|_{W^{1,2}}^2\leq C\left(\|d\hat{F}_E\|_{L^2(D_{\epsilon}(p))}^2-\|dF_E\|_{L^2(\Gamma_L)}^2+\epsilon^2\|dF\|_{L^2(\Gamma_L)}\right).
\end{equation}

Note that, as an easy consequence of \eqref{h1close}, we have
\[
\|dF_E\|_{L^2(\Gamma_L)}^2\leq \|d\hat{F}_E\|_{L^2(D_{\eps})}^2+C\eps^2\|dF\|_{L^2(\Gamma_L)},
\]
while by definition of $\hat{F}=T(F)$, Lemma \ref{ext.bds} gives
\[
\|d\hat{F}_E\|_{L^2(D_{\eps})}^2\leq (1+Ce^{-2L})\|dF_E\|_{L^2(\Gamma_L)}^2,
\]
and combining these estimates yields
\begin{eqnarray*}
\|d\hat{F}_E\|_{L^2(D_{\eps})}^2-\|dF_E\|_{L^2(\Gamma_L)}^2&\leq &Ce^{-2L}\|dF_E\|_{L^2(\Gamma_L)}^2\\
&\leq &Ce^{-2L}\left(\|d\hat{F}_E\|_{L^2(D_\epsilon)}^2+\eps^2\|dF\|_{L^2(\Gamma_L)}\right).
\end{eqnarray*}
In particular, we can recast estimate \eqref{h1close} as
\begin{equation}\label{h1close.2}
\|\hat{F}-\Phi\|_{W^{1,2}}^2\leq C'e^{-2L}\|d\hat{F}_E\|_{L^2(D_{\eps})}^2+C'\eps^2\|dF\|_{L^2(\Gamma_L)}.
\end{equation}

Now, since the components of $\Phi$ belong to the (finite-dimensional) first eigenspace $\mathcal{E}_{\lambda_1}(M,g_0)$, note that with respect to either of the reference metrics $g_1$ or $g_p$, we have uniform $C^3$ estimates of the form
\begin{equation}\label{c3.bd}
\|\Phi\|_{C^3}\leq C \|\Phi\|_{W^{1,2}}\leq C'\|\hat{F}\|_{W^{1,2}}\leq C''.
\end{equation}
As  consequence, decomposing $\Phi$ into its even and odd components
\[
\Phi_O(z)=\frac{1}{2}\left(\Phi(z)-\Phi(-z)\right)=d\Phi(0) \cdot z+O(|z|^3)
\]
and
\[
\Phi_E(z)=\frac{1}{2}\left(\Phi(z)+\Phi(-z)\right)=\Phi(0)+O(|z|^2),
\]
on the disk $D_{\eps}(p)\cong D_{\eps}(0)$, and observing that $d\Phi_E(0)=0$, we see that 
$$|d\Phi_E(z)|_{g_p}\leq \|\Phi_E\|_{C^2}|z|\leq C|z|,$$
and therefore
$$\|d\Phi_E\|_{L^2(D_{\epsilon}(p))}^2\leq C\epsilon^4.$$
As a consequence, it follows that
\begin{eqnarray*}
\|d\hat{F}_E\|_{L^2(D_{\eps}(p))}^2&\leq &2\|d\Phi_E\|_{L^2(D_{\eps}(p))}^2+2\|\Phi-\hat{F}\|_{W^{1,2}}^2\\
&\leq & 2\|\Phi-\hat{F}\|_{W^{1,2}}^2+C\epsilon^4,
\end{eqnarray*}
and combining this with \eqref{h1close.2}, we deduce that
$$\|\hat{F}-\Phi\|_{W^{1,2}}^2\leq C''e^{-2L}(\|\hat{F}-\Phi\|_{W^{1,2}}^2+\epsilon^4)+C'\eps^2\|dF\|_{L^2(\Gamma_L)}.$$
Thus, fixing $L=L_0$ sufficiently large (independent of $\epsilon$) so that $C''e^{-2L}<\frac{1}{2}$, we arrive at an estimate of the form
\begin{equation}\label{h1close.3}
\|\hat{F}-\Phi\|_{W^{1,2}}^2\leq C(\epsilon^2\|dF\|_{L^2(\Gamma_L)}+\epsilon^4).
\end{equation}

Next, by an application of \eqref{c3.bd} and \eqref{h1close.3}, we observe that
\begin{eqnarray*}
\|d\hat{F}\|_{L^2(D_{\eps}(p))}^2&\leq & 2\|d\Phi\|_{L^2(D_{\eps}(p))}^2+2C(\eps^2\|dF\|_{L^2(\Gamma_L)}+\epsilon^4)
\leq C'\epsilon^2,
\end{eqnarray*}
which together with the fact that 
$$
\|d\hat{F}\|_{L^2(M)}^2\geq \lambda_1\|\hat{F}\|_{L^2(M,g_0)}^2\geq \bar{\lambda}_1(M,g_0)-C\epsilon^2
$$
forces
$$
\|dF\|_{L^2(M\setminus D_{\eps}(p))}^2\geq \bar{\lambda}_1(M,g_0)-C'\epsilon^2.
$$
On the other hand, since
$$
\|dF\|_{L^2(M',g')}^2\leq \Lambda_1(M',[g'])\leq \bar{\lambda}_1(M,g_0)
$$
by assumption, it follows that
$$
\|dF\|_{L^2(\Gamma_L)}^2=\|dF\|_{L^2(M',g')}^2-\|dF\|_{L^2(M\setminus D_{\eps}(p))}^2\leq C'\epsilon^2.
$$

Applying this estimate to the right-hand side of \eqref{h1close.3}, we arrive at the simple bound
$$
\|\hat{F}-\Phi\|_{W^{1,2}}^2\leq C'\epsilon^3,
$$
and in particular, we have
\begin{equation}\label{odd.small}
\|d\hat{F}_O-d\Phi_O\|_{L^2(D_{\eps}(p))}^2\leq C'\epsilon^3
\end{equation}
for the odd components $\hat{F}_O$, $\Phi_O$ on $D_{\eps}(p)$. On the other hand, it follows from \eqref{loc.even} that $\hat{F}_O\equiv 0$ on $D_{e^{-L}\eps}(p)$, which together with \eqref{odd.small} implies
$$
\|d\Phi_O\|_{L^2(D_{e^{-L}\eps}(p))}^2\leq C'\epsilon^3.
$$
Recalling that $L=L_0$ was fixed above, and appealing to the uniform $C^3$ estimates \eqref{c3.bd} for $\Phi$ to see that
$$
|d\Phi_O(z)-d\Phi(p)|\leq C|z-p|\leq C\epsilon
$$
for all $z\in D_{\eps}(p)$, we then deduce that
$$
|d\Phi(p)|^2\leq C\epsilon^2+\frac{C'}{\epsilon^2}\int_{D_{e^{-L_0}\eps}(p)}|d\Phi_O|^2\leq C''\epsilon.
$$
Observing moreover that \eqref{h1close.3} easily leads to an estimate of the form $\int_M(1-|\Phi|)^2\leq C\epsilon^2$, we conclude that $\Phi_{\epsilon}=\Phi$ satisfies the conclusions of Lemma \ref{cap.grad.small}. 

\end{proof}

\section{Attaching handles}

In this section, we prove the following analog of Theorem \ref{cap.gap} in the orientable setting; in this case, an alternate proof can also be found in~\cite{PetridesNew}. The key difference between the arguments that follow and those of the previous section--and the main source of improvement over the methods of~\cite{KKMS}--is the central role played by the choice of the attaching map used to affix a handle at two adjacent points.

\begin{theorem}\label{hand.gap}
Let $M$ be a closed surface, and let $g_0$ be a (possibly degenerate) metric on $M$ of the form $g_0=fg_1$, where $g_1$ is a smooth metric and $f$ is a smooth, nonnegative function with $\#f^{-1}\{0\}<\infty$. Then
$$\bar{\lambda}_1(M,g_0)=\bar{\lambda}_1([g_1],f\,dv_{g_1})<\Lambda_1(M\#\mathbb{T}^2).$$
\end{theorem}

As in the previous section, we begin by fixing an arbitrary point $p\in M$, and assume without loss of generality--after making a smooth conformal change--that the reference metric $g_1$ is Euclidean on some geodesic disk $D_{\delta_0}(p)$ about $p$ of radius $\delta_0<\frac{1}{2}$. Next, fix a unit vector $v\in T_pM$, and for each $\epsilon<\frac{1}{4}\delta_0^2$, set 
$$q=q_{\epsilon}:=p+\sqrt{\epsilon}v,$$
so that the disks $D_{\epsilon}(p)$ and $D_{\epsilon}(q_{\epsilon})$ are disjoint and contained in $D_{\delta_0}(p)$. Denoting their union by
$$\mathcal{D}_{\epsilon}:=D_{\epsilon}(p)\cup D_{\epsilon}(q),$$
we then define
$$N_{p,v,\epsilon}:=M\setminus \mathcal{D}_{\epsilon},$$
and consider the orientation-reversing involution
$$\iota: \mathcal{D}_{\epsilon}\to \mathcal{D}_{\epsilon}$$
given by
$$\iota(x)=q+A_v(x-p)\text{ for }x\in D_{\epsilon}(p)\text{ and }\iota(x)=p+A_v(x-q)\text{ for }x\in D_{\epsilon}(q),$$
where $A_v\in \Isom(T_pM)$ is given by reflection across $\mathrm{Span}(v)$, so that $Av=v$ and $Av^{\perp}=-Av^{\perp}$.

Next, denote by $T_L=\Sph^1\times [-L,L]$ the flat cylinder of width $2\pi$ and length $2L$, and define an attaching map
$$
\phi\colon \partial T_L\to \partial N_{p,v,\epsilon}
$$
by
$$\phi(z,-L)=p+\epsilon z\text{ and }\phi(z,L)=\iota(p+\epsilon z).$$
Proceeding again as in Section \ref{glue}, consider the $L^{\infty}$ metric $g'_{\epsilon,L}$ defined on 
$$M'=N_{p,v,\epsilon}\cup_{\phi}T_L\approx M\#\mathbb{T}^2$$
such that $g'_{\epsilon,L}=g_1$ on $N_{p,v,\epsilon}$ and $g'_{\epsilon,L}$ agrees with the given flat metric on $T_L$. We then have the following analog of Proposition \ref{cap.prop}.

\begin{proposition}\label{hand.prop}
If $\Lambda_1(M',[g'_{\epsilon,L}])\leq \bar{\lambda}_1(M,g_0)$ for all $L\in (0,\infty)$ and $\epsilon<\frac{\delta_0^2}{4}$, then $f(p)=0$.
\end{proposition}

Once Proposition \ref{hand.prop} has been proved, Theorem \ref{hand.gap} once again follows immediately by selecting a point $p$ for which $f(p)>0$ and choosing $\epsilon,L$ such that $\Lambda_1(M',[g'_{\epsilon,L}])>\bar{\lambda}_1(M,g_0)$. The rest of this section is therefore devoted to the proof of Proposition \ref{hand.prop}.

\subsection{A refined extension operator for handles} 

Let $T_L=\mathbb{S}^1\times [-L,L]$ and $\mathcal{D}_{\epsilon}=D_{\epsilon}(p)\cup D_{\epsilon}(q)$ be as above, equipped with the involution $\iota\colon \mathcal{D}_{\epsilon}\to \mathcal{D}_{\epsilon}$ and attaching map $\phi\colon \partial T_L\to \partial \mathcal{D}_{\epsilon}$ satisfying $\phi\circ \rho=\iota\circ \phi$ for the reflection $\rho(z,t)=(z,-t)$.

Similar to the cross-cap case, we next introduce an improved extension operator $K\colon W^{1,2}(T_L)\to W^{1,2}(\mathcal{D}_{\epsilon})$ by distinguishing between the even functions
$$
\mathcal{E}:=\{u\in W^{1,2}(T_L)\mid u(z,-t)=u(z,t)\}
$$
and odd functions
$$
\mathcal{O}:=\{u\in W^{1,2}(T_L)\mid u(z,-t)=-u(z,t)\}
$$
with respect to the reflection $\rho$. In this setting, note once again that $\mathcal{E}\oplus \mathcal{O}$ gives an orthogonal decomposition of $W^{1,2}(T_L)$, and if $u\in \mathcal{O}$, then $u$ vanishes on $\Sph^1\times\{0\}$. Denoting by $\pi_E\colon W^{1,2}(T_L)\to \mathcal{E}$ and $\pi_O\colon W^{1,2}(T_L)\to \mathcal{O}$ the associated projections, we then have the following.

\begin{lemma}\label{hand.ext.bds}
For $L\geq \frac{3}{2}\log(2)$, there exists an operator $K\colon W^{1,2}(T_L)\to W^{1,2}(\mathcal{D}_{\epsilon})$ such that
$$
K(u)\circ \phi=u\text{ on }\partial T_L,
$$
$$
\|K(u)\|_{L^{\infty}(\mathcal{D}_{\epsilon})}\leq 2\|u\|_{L^{\infty}(\mathcal{D}_{\epsilon})},
$$
and writing $u_E=\pi_E(u)$ and $u_O=\pi_O(u)$, we have
$$
\|d[K(u_E)]\|_{L^2(\mathcal{D}_{\epsilon})}^2\leq (1+Ce^{-2L})\|d(u_E)\|_{L^2(T_L)}^2
$$
and
$$
\|d[K(u_O)\|_{L^2(\mathcal{D}_{\epsilon})}^2=\|d(u_O)\|_{L^2(T_L)}^2.
$$
Moreover, we have
$$
K(u_E)\circ \iota=K(u_E),
$$
$$
K(u_O)\circ \iota =-K(u_O),
$$
and $K(u_O)\equiv 0$ on $D_{ e^{-L}\epsilon}(p)\cup D_{ e^{-L}\epsilon}(q)$.
\end{lemma}
\begin{proof}
The construction is similar to that of Lemma \ref{ext.bds}. As in the proof of Lemma \ref{ext.bds}, note first that the inverse of the attaching map $\phi^{-1}\colon \partial \mathcal{D}_{\epsilon}\to \partial T_L$ extends naturally to a conformal map
$$
F: [D_{\epsilon}(p)\setminus D_{e^{-L}\epsilon}(p)]\cup [D_{\epsilon}(q)\setminus D_{e^{-L}\epsilon}(q)]\to T_L\setminus (\Sph^1\times \{0\})
$$
satisfying 
\begin{equation}\label{f.equiv}
F\circ \iota=\rho \circ F.
\end{equation}
For $u\in \mathcal{O}$, we then define $K(u)\in W^{1,2}(\mathcal{D}_{\epsilon})$ by setting
$$
K(u)=u\circ F\text{ on }[D_{\epsilon}(p)\setminus D_{e^{-L}\epsilon}(p)]\cup [D_{\epsilon}(q)\setminus D_{e^{-L}\epsilon}(q)]
$$
and
$$
K(u)=0\text{ on }D_{e^{-L}\epsilon }(p)\cup D_{ e^{-L}\epsilon}(q),
$$
using crucially the fact that $u\circ F=0$ on $\partial D_{e^{-L}\epsilon}(p)\cup \partial D_{e^{-L}\epsilon}(q)$. It is then clear that $K\colon\mathcal{O}\to W^{1,2}(\mathcal{D}_{\epsilon})$ preserves the $L^{\infty}$ norm as well as the Dirichlet energy, by the conformality of $F$, and the condition $K(u)\circ \iota=-K(u)$ is an immediate consequence of \eqref{f.equiv}. 

For a general $u=u_E+u_O$, we then define
$$
K(u):=H(u_E)+K(u_O),
$$
where $H(u_E)$ denotes the harmonic extension of $u_E\circ \phi^{-1}\in W^{1/2,2}(\partial\mathcal{D}_{\epsilon})$ into $\mathcal{D}_{\epsilon}$, for which the desired estimates follow from Lemma \ref{hex.lem}, exactly as in the proof of Lemma \ref{ext.bds}. Moreover, since $u_E\circ \phi^{-1}=u_E\circ \rho \circ \phi^{-1}\circ \iota=u_E\circ \phi^{-1}\circ \iota$ on $\partial \mathcal{D}_{\epsilon}$, it follows that $K(u_E)\circ \iota=K(u_E)$, as desired.
\end{proof}

\begin{remark}\label{h.vect.valued}
Once again, note that the Sobolev spaces in Lemma \ref{hand.ext.bds} could refer to the vector-valued Sobolev spaces $W^{1,2}(T_L,\mathbb{R}^n)$ and $W^{1,2}(\mathcal{D}_{\epsilon},\mathbb{R}^n)$ for arbitrary $n$, with no effect on the conclusions of Lemma \ref{ext.bds}.
\end{remark}

In particular, for $(M',g'_{\epsilon,L})$ defined as above, we can then define a map
$$
S\colon W^{1,2}(M')\to W^{1,2}(M)
$$
by setting
$$
S(u)(x)=u(x)\text{ for }x\in N_{p,v,\epsilon}
$$
and
$$
S(u)(x)=K(u|_{T_L})(x)\text{ for }x\in T_L.
$$
It then follows from Lemma \ref{hand.ext.bds} that 
\begin{equation}\label{s.ener.drop}
\|d(S(u))\|_{L^2(M)}^2=\|du\|_{L^2(M')}^2+\|d S(u)_E\|_{L^2(\mathcal{D}_{\epsilon})}^2-\|du_E\|_{L^2(T_L)}^2,
\end{equation}
where $S(u)_E:=\frac{1}{2}(u+u\circ \iota)$ on $\mathcal{D}_{\epsilon}$ and $u_E=\pi_E(u|_{T_L})$ as above, and
\begin{equation}\label{s.loc.even}
S(u)=S(u)_E\text{ on }D_{e^{-L}\epsilon}(p)\cup D_{e^{-L}\epsilon}(q).
\end{equation}

\subsection{Proof of Proposition \ref{hand.prop}} As in the cross-cap case, we observe next that the proof of Proposition \ref{hand.prop} can be reduced to the following more quantitative lemma.

\begin{lemma}\label{hand.grad.small}
In the setting of Proposition \ref{hand.prop}, there exists $L=L_0\in (1,\infty)$ and $C<\infty$ independent of $\epsilon\in \left(0,\frac{\delta_0^2}{4}\right)$ such that if $\Lambda_1(M',[g'_{\epsilon,L}])\leq \bar{\lambda}_1(M,g_0)$, then there exists a map $\Phi_{\epsilon}:M\to \mathbb{R}^{n+1}$ satisfying
$$
\Delta_{g_1}\Phi_{\epsilon}=\lambda_1(M,g_0)f \Phi_{\epsilon},
$$
$$
\|1-|\Phi_{\epsilon}|\|_{L^2(M,g_1)}^2\leq C\epsilon^2,
$$
and
$$
|d\Phi_{\epsilon}(p)|^2\leq C\epsilon|\log\epsilon|.
$$
\end{lemma}

With Lemma \ref{hand.grad.small}, Proposition \ref{hand.prop} follows with no additional work.

\begin{proof}[Proof of Proposition \ref{hand.prop}]
The same argument used in Section 3.2 to deduce Proposition \ref{cap.prop} from Lemma \ref{cap.grad.small} applies verbatim to show that Lemma \ref{hand.grad.small} implies Proposition \ref{hand.prop}.
\end{proof}

It remains to prove Lemma \ref{hand.grad.small}. The proof is broadly similar to that of Lemma \ref{cap.grad.small}, though relating evenness and oddness with respect to the involution $\iota\colon \mathcal{D}_{\epsilon}\to \mathcal{D}_{\epsilon}$ to gradient estimates is slightly more delicate.

\begin{proof}[Proof of Lemma \ref{hand.grad.small}]
Similarly to the proof of Lemma~\ref{cap.grad.small}, the map $F$ and the objects derived from it depend on the choice of $\epsilon$ and $L$. We do not include that in the notation for these objects, but keep careful track on the dependence on the parameters in all the estimates.

To begin, fix a small but arbitrary $\epsilon>0$ and some $L\geq \frac{3}{2}\log(2)$, which by the end of the proof we will set to be a fixed constant $L=L_0$ independent of $\epsilon>0$. Assuming that
$$
\Lambda_1(M',[g'_{\epsilon,L}]\leq \bar{\lambda}_1(M,g_0),
$$
we apply Proposition \ref{balanced.f} to the functional $\beta\in W^{1,2}(M')^*$ given by
$$
\beta(u):=\int_M S(u)\, dv_{g_0},
$$
to deduce the existence of a map $F\in W^{1,2}(M',\Sph^n)$ for some $n\in \mathbb{N}$ satisfying
\begin{equation}\label{f.ener.2}
\int_{M'}|dF|_{g'}^2\,dv_{g'}\leq \Lambda_1(M',[g'])\leq \bar{\lambda}_1(M,g_0),
\end{equation}
such that $\hat{F}=S(F)\in W^{1,2}(M,\mathbb{R}^{n+1})$ satisfies
\begin{equation}\label{hand.bal}
\int_M\hat{F}\,dv_{g_0}=0.
\end{equation}
By Lemma \ref{hand.ext.bds} and the definition of $S$, we note again that $|\hat{F}|=|F|=1$ on $M\setminus \mathcal{D}_{\epsilon}$, while $|\hat{F}|\leq 2$ on $\mathcal{D}_{\epsilon}$.

From here, we can argue exactly as in the proof of Lemma \ref{cap.grad.small}, using \eqref{s.ener.drop} in place of \eqref{f.ener.bd}, to arrive at an estimate of the form
\begin{equation}\label{h1close.o}
\|\hat{F}-\Phi\|_{W^{1,2}_{g_1}}^2\leq C\left(\|d\hat{F}_E\|_{L^2(\mathcal{D}_{\epsilon})}^2-\|dF_E\|_{L^2(T_L)}^2+\epsilon^2\|dF\|_{L^2(T_L)}\right),
\end{equation}
where $C=C(M,g_1)$ is a constant independent of $\epsilon$ and $L$, $\hat{F}_E:=\frac{1}{2}(F+F\circ \iota)$ on $\mathcal{D}_{\epsilon}$, and $\Phi$ denotes the projection of $\hat{F}$ onto the first eigenspace for $\Delta_{g_0}$, in the sense of Remark \ref{eigen.basics}. Moreover, combining the estimates of Lemma \ref{hand.ext.bds} with the nonnegativity of the right-hand-side of \eqref{h1close.o}, an argument identical to the derivation of \eqref{h1close.2} in the previous section yields an estimate of the form
\begin{eqnarray*}
\|\hat{F}-\Phi\|_{W^{1,2}}^2&\leq & C' e^{-2L}\|d\hat{F}_E\|_{L^2(\mathcal{D}_{\epsilon})}^2+C'\epsilon^2\|dF\|_{L^2(T_L)}\\
&\leq &Ce^{-2L}(\|\hat{F}-\Phi\|^2_{W^{1,2}}+\|d\Phi_E\|_{L^2(\mathcal{D}_{\epsilon})}^2)+C'\epsilon^2\|dF\|_{L^2(T_L)}.
\end{eqnarray*}
We now fix $L=L_0$ large enough such that $Ce^{-2L}<\frac{1}{2}$, and rearrange this inequality to obtain the estimate
\begin{equation}\label{or.ineq.1}
\|\hat{F}-\Phi\|_{W^{1,2}}^2\leq C\|d\Phi_E\|_{L^2(\mathcal{D}_{\epsilon})}^2+C\epsilon^2\|dF\|_{L^2(T_L)}.
\end{equation}
Note that the uniform $C^3$ estimates \eqref{c3.bd} continue to hold in the present setting, so that
\begin{equation}\label{c3.o}
\|\Phi\|_{C^3}\leq C'\|\hat{F}\|_{W^{1,2}}\leq C''
\end{equation}
with all constants independent of $\epsilon>0$. As an immediate consequence, we have the non-sharp bound
$$
\|d\Phi\|_{L^2(\mathcal{D}_{\epsilon})}^2\leq C\epsilon^2,
$$
which together with \eqref{or.ineq.1} implies that
$$
\|\hat{F}-\Phi\|_{W^{1,2}}^2\leq C\epsilon^2,
$$
and therefore
$$
\|d\hat{F}\|_{L^2(\mathcal{D}_{\epsilon})}^2\leq 2\|\hat{F}-\Phi\|_{W^{1,2}}^2+\|d\Phi\|_{L^2(\mathcal{D}_{\epsilon})}^2\leq C'\epsilon^2.
$$
In particular, together with the nonnegativity of the right-hand side of \eqref{h1close.o}, this implies that
\begin{eqnarray*}
\|dF\|_{L^2(T_L)}^2&=&\|dF_O\|_{L^2(T_L)}^2+\|dF_E\|_{L^2(T_L)}^2\\
&=&\|d\hat{F}_O\|_{L^2(\mathcal{D}_{\epsilon})}^2+\|dF_E\|_{L^2(T_L)}^2\\
&\leq &\|d\hat{F}_O\|_{L^2(\mathcal{D}_{\epsilon})}^2+\|d\hat{F}_E\|_{L^2(\mathcal{D}_{\epsilon})}^2+\epsilon^2\|dF\|_{L^2(T_L)}\\
&\leq &C'\epsilon^2+\epsilon^2\|dF\|_{L^2(T_L)},
\end{eqnarray*}
giving us an estimate of the form
\begin{equation}\label{tube.grad.bd}
\|dF\|_{L^2(T_L)}^2\leq C''\epsilon^2.
\end{equation}
Applying this estimate to the right-hand side of \eqref{or.ineq.1}, we then see that
\begin{equation}\label{h1.o2}
\|\hat{F}-\Phi\|_{W^{1,2}}^2\leq C\left(\|d\Phi_E\|_{L^2(\mathcal{D}_{\epsilon})}^2+\epsilon^3\right)\leq C'\epsilon^2.
\end{equation}
As an immediate consequence, we see that
\begin{equation}\label{phi1.l2.diff}
\|1-|\Phi|\|_{L^2(M,g_1)}^2\leq C\epsilon^2.
\end{equation}

Next, recalling that $L=L_0$ has already been fixed, a simple combination of \eqref{tube.grad.bd} with the fundamental theorem of calculus yields an estimate of the form
$$
\left|\int_{\Sph^1}F(\cdot,L)-F(\cdot,-L)\right|\leq C\|dF\|_{L^2(T_L)}\leq C'\epsilon,
$$
which is evidently equivalent to the bound
$$
\frac{1}{\epsilon}\left|\int_{\partial D_{\epsilon}(p)}\hat{F}\,dv_{g_p}-\int_{\partial D_{\epsilon}(q)}\hat{F}\,dv_{g_p}\right|\leq C'\epsilon
$$
for $\hat{F}$. On the other hand, applying [KKMS, Lemma 8.15] with $u=|\hat{F}-\Phi|$, we also have an estimate of the form
$$
\frac{1}{\epsilon}\left|\int_{\partial \mathcal{D}_{\epsilon}}|\hat{F}-\Phi|\,dv_{g_p}\right|\leq C\sqrt{|\log\epsilon|}\|\hat{F}-\Phi\|_{W^{1,2}(M)},
$$
and combining these two estimates with \eqref{h1.o2}, we deduce that
\begin{eqnarray*}
\left|\frac{1}{\epsilon}\int_{\partial D_{\epsilon}(p)}\Phi\,dv_{g_p}-\frac{1}{\epsilon}\int_{\partial D_{\epsilon}(q)}\Phi\,dv_{g_p}\right|&\leq& C\epsilon+C\sqrt{|\log\epsilon|}\|\hat{F}-\Phi\|_{W^{1,2}(M)}\\
&\leq &C'\epsilon\sqrt{|\log\epsilon|}.
\end{eqnarray*}

Now, by \eqref{c3.o}, note that we have
$$
|\Phi(z)-\Phi(p)|\leq C|z-p|\leq C\epsilon\text{ on }\partial D_{\epsilon}(p)
$$
and $|\Phi(z)-\Phi(q)|\leq C\epsilon$ on $\partial D_{\epsilon}(q)$, which together with the preceding integral estimate yields an estimate of the form
$$
|\Phi(p)-\Phi(q)|\leq C\epsilon\sqrt{|\log\epsilon|}.
$$
On the other hand, by another application of \eqref{c3.o}, note that $\Phi$ satisfies an estimate of the form
$$
|\Phi(p)+d\Phi(p)\cdot (z-p)-\Phi(z)|\leq C|z-p|^2
$$
for all $z\in D_{\delta_0}(p)$, and applying this with $z=q=p+\sqrt{\epsilon}v$ gives
$$
|\sqrt{\epsilon} d\Phi(p)\cdot v+\Phi(p)-\Phi(q)|\leq C\epsilon.
$$
In particular, since $|\Phi(p)-\Phi(q)|\leq C\epsilon\sqrt{|\log\epsilon|}$, it follows that
$$
\sqrt{\epsilon}|d\Phi(p)\cdot v|\leq C\epsilon+|\Phi(p)-\Phi(q)|\leq C'\epsilon\sqrt{|\log\epsilon|},
$$
and therefore
\begin{equation}\label{dphi.v}
|d\Phi(p)\cdot v|\leq C'\sqrt{\epsilon|\log\epsilon|}.
\end{equation}
Moreover, using $q$ as a basepoint instead of $p$, the same argument gives $|d\Phi(q)\cdot v|\leq C'\sqrt{\epsilon|\log\epsilon|}$ as well, and by another application of the $C^3$ estimate \eqref{c3.o}, it follows that
\begin{equation}\label{dphi.v.d}
|d\Phi(x)\cdot v|\leq C''\sqrt{\epsilon|\log\epsilon|}\text{ for all }x\in \mathcal{D}_{\epsilon}.
\end{equation}

To complete the proof the lemma, we need a matching estimate for $d\Phi(p)\cdot v^{\perp}$, where $v^{\perp}$ is a unit vector orthogonal to $v$, and this is where the choice of attaching map plays a key role. To this end, recall that the involution $\iota\colon \mathcal{D}_{\epsilon}\to \mathcal{D}_{\epsilon}$ is given by
$$
\iota(x)=q+A(x-p)\text{ for }x\in D_{\epsilon}(p)
$$
and
$$
\iota(x)=p+A(x-q)\text{ for }x\in D_{\epsilon}(q),
$$
where $A$ is the reflection for which $Av=v$ and $Av^{\perp}=-v^{\perp}$. For any $x\in \mathcal{D}_{\epsilon}$, the even part $\Phi_E=\frac{1}{2}[\Phi+\Phi\circ \iota]$ then satisfies
$$
d\Phi_E(x)=\frac{1}{2}[d\Phi(x)+d\Phi(\iota(x))\cdot A],
$$
and it follows from \eqref{dphi.v.d} that
$$
|d\Phi_E(x)\cdot v|=\frac{1}{2}|d\Phi(x)\cdot v+d\Phi(\iota(x))\cdot v|\leq C''\sqrt{\epsilon|\log\epsilon|}
$$
for all $x\in \mathcal{D}_{\epsilon}$. On the other hand, in the $v^{\perp}$ direction, we see that
$$
d\Phi_E(x)\cdot v^{\perp}=\frac{1}{2}[d\Phi(x)-d\Phi(\iota(x))]\cdot v^{\perp},
$$
while another application of the $C^3$ estimates \eqref{c3.o} gives an estimate of the form
$$
|d\Phi(x)-d\Phi(\iota(x))|\leq C|x-\iota(x)|\leq 2C\sqrt{\epsilon}
$$
for all $x\in \mathcal{D}_{\epsilon}$, so that
$$
|d\Phi_E(x)\cdot v^{\perp}|\leq C\sqrt{\epsilon}
$$
for all $x\in \mathcal{D}_{\epsilon}$ as well. Putting these estimates together, we then see that
$$
|d\Phi_E(x)|\leq C'\sqrt{\epsilon|\log\epsilon|}
$$
for all $x\in \mathcal{D}_{\epsilon}$, and integrating over $\mathcal{D}_{\epsilon}$ yields
\begin{equation}\label{dphi.e.tiny}
\|d\Phi_E\|_{L^2(\mathcal{D}_{\epsilon})}^2\leq C\epsilon^3|\log\epsilon|.
\end{equation}

Finally, recall from Lemma \ref{hand.ext.bds} and the definition of $\hat{F}$ that $\hat{F}_O\equiv 0$ on $D_{e^{-L}\epsilon}(p)\cup D_{e^{-L}\epsilon}(q)$, so that
$$
\|d\Phi_O\|_{L^2(D_{e^{-L}\epsilon}(p))}^2\leq \|\hat{F}-\Phi\|_{W^{1,2}}^2.
$$
Using \eqref{h1.o2} and \eqref{dphi.e.tiny} to estimate the right-hand side, we then see that
$$
\|d\Phi_O\|_{L^2(D_{e^{-L}\epsilon}(p))}^2\leq C\left(\|d\Phi_E\|_{L^2(\mathcal{D}_{\epsilon})}^2+\epsilon^3\right)\leq C'\epsilon^3|\log\epsilon|,
$$
and summing with \eqref{dphi.e.tiny} gives
$$
\|d\Phi\|_{L^2(D_{e^{-L}\epsilon}(p))}^2\leq C\epsilon^3|\log\epsilon|.
$$
Recalling once again that $L=L_0$ is a fixed constant independent of $\epsilon$, it follows from the $C^3$ estimates \eqref{c3.o} that
\begin{equation}\label{fin.grad.bd}
|d\Phi(p)|^2\leq \frac{\|d\Phi\|_{L^2(D_{e^{-L}\epsilon}(p))}^2}{|D_{e^{-L}\epsilon}(p)|}+C\epsilon^2\leq C'\epsilon|\log\epsilon|.
\end{equation}
Writing $\Phi_{\epsilon}=\Phi$, we then see that \eqref{fin.grad.bd} and \eqref{phi1.l2.diff} together give the desired estimates, completing the proof of the lemma.

\end{proof}

\section{Proof of Theorem~\ref{gap.thm}}

With Theorems~\ref{cap.gap} and~\ref{hand.gap} in place, Theorem~\ref{gap.thm} follows form Theorem~\ref{gap.suff} via a standard induction argument, see e.g.~\cite{Petrides1, PetridesNew, KKMS}. We reproduce this argument below.

We prove Theorem~\ref{gap.thm} by induction on $\chi(M)$. If $\chi(M) = 2$, then $M$ is a $2$-sphere and the round metric is $\bar\lambda_1$-maximizing by the classical result of Hersch~\cite{Hersch}. Choosing $g_0$ to be the round metric in Theorems~\ref{cap.gap} and~\ref{hand.gap} concludes the proof of the base of induction. Suppose now that the theorem is proved for all $M_0$ with $\chi(M_0)>\chi(M)$. Note that if $M\approx M_0\# \mathbb{RP}^2$ or $M\approx M_0\# \T^2$, then $\chi(M_0)>\chi(M)$, so that by the induction hypothesis there exists a $\bar\lambda_1$-maximizing metric $g_0$ on $M_0$, i.e. $\bar\lambda_1(M_0,g_0) = \Lambda_1(M_0)$. Applying Theorem~\ref{cap.gap} or~\ref{hand.gap} to that metric $g_0$ yields $\Lambda_1(M_0) = \bar\lambda_1(M_0,g_0)<\Lambda_1(M)$ and, as a result, there exists a $\bar\lambda_1$-maximizing metric on $M$ by Theorem~\ref{gap.suff}. The proof is complete.

\end{document}